\documentclass[a4paper,10pt]{article}

\usepackage{amsmath,amssymb,amsfonts}
\usepackage{ifthen}
\usepackage{hyperref}
\usepackage[amsmath,thmmarks,hyperref]{ntheorem}

\theoremstyle{plain}
\theoremseparator{.}
\newtheorem{theorem}{Theorem}[section]
\newtheorem{proposition}[theorem]{Proposition}
\newtheorem{lemma}[theorem]{Lemma}
\newtheorem{corollary}[theorem]{Corollary}

\theorembodyfont{}
\theoremsymbol{\ensuremath{{}_\blacksquare}}
\newtheorem{definition}[theorem]{Definition}

\theoremheaderfont{\itshape}
\newtheorem{remark}[theorem]{Remark}

\theoremsymbol{\ensuremath{\square}}
\newtheorem*{proof}{Proof}

\newcommand{\eps}{\varepsilon}
\DeclareMathOperator*{\argmin}{arg\,min}
\DeclareMathOperator{\range}{Ran}
\newcommand{\field}[1]{\mathbb{#1}}
\newcommand{\R}{\field{R}}

\newcommand{\inner}[3][n]{\SwitchBracketsizeLeft{#1}\LeftBracketSize\langle#2,#3\SwitchBracketsizeRight{#1}\RightBracketSize\rangle}
\newcommand{\abs}[2][n]{\SwitchBracketsizeLeft{#1}\LeftBracketSize\lvert#2\SwitchBracketsizeRight{#1}\RightBracketSize\rvert}
\newcommand{\norm}[2][n]{\SwitchBracketsizeLeft{#1}\LeftBracketSize\lVert#2\SwitchBracketsizeRight{#1}\RightBracketSize\rVert}
\newcommand{\set}[3][b]{\SwitchBracketsizeLeft{#1}\LeftBracketSize\{#2:#3\SwitchBracketsizeRight{#1}\RightBracketSize\}}

\newcommand{\NextScriptStyle}[1]{{\scriptstyle{#1}}}
\newcommand{\NextScriptScriptStyle}[1]{{\scriptscriptstyle{#1}}}
\newcommand{\NextTextStyle}[1]{{\textstyle{#1}}}
\newcommand{\NextDisplayStyle}[1]{{\displaystyle{#1}}}
\newcommand{\SwitchBracketsizeLeft}[1]{
  \ifthenelse{\equal{#1}{b}\OR\equal{#1}{big}}{\let\LeftBracketSize=\bigl}{
    \ifthenelse{\equal{#1}{B}\OR\equal{#1}{Big}}{\let\LeftBracketSize=\Bigl}{
      \ifthenelse{\equal{#1}{g}\OR\equal{#1}{bigg}}{\let\LeftBracketSize=\biggl}{
    \ifthenelse{\equal{#1}{G}\OR\equal{#1}{Bigg}}{\let\LeftBracketSize=\Biggl}{
      \ifthenelse{\equal{#1}{s}\OR\equal{#1}{small}}{\let\LeftBracketSize=\NextScriptStyle}{
        \ifthenelse{\equal{#1}{ss}}{\let\LeftBracketSize=\NextScriptScriptStyle}{
          \ifthenelse{\equal{#1}{t}\OR\equal{#1}{text}}{\let\LeftBracketSize=\NextTextStyle}{
        \ifthenelse{\equal{#1}{d}\OR\equal{#1}{display}}{\let\LeftBracketSize=\NextDisplayStyle}{
          \ifthenelse{\equal{#1}{a}\OR\equal{#1}{auto}}{\let\LeftBracketSize=\left}{
            \let\LeftBracketSize=\relax}}}}}}}}}}
\newcommand{\SwitchBracketsizeRight}[1]{
  \ifthenelse{\equal{#1}{b}\OR\equal{#1}{big}}{\let\RightBracketSize=\bigr}{
    \ifthenelse{\equal{#1}{B}\OR\equal{#1}{Big}}{\let\RightBracketSize=\Bigr}{
      \ifthenelse{\equal{#1}{g}\OR\equal{#1}{bigg}}{\let\RightBracketSize=\biggr}{
    \ifthenelse{\equal{#1}{G}\OR\equal{#1}{Bigg}}{\let\RightBracketSize=\Biggr}{
      \ifthenelse{\equal{#1}{s}\OR\equal{#1}{small}}{\let\RightBracketSize=\NextScriptStyle}{
        \ifthenelse{\equal{#1}{ss}}{\let\RightBracketSize=\NextScriptScriptStyle}{
          \ifthenelse{\equal{#1}{t}\OR\equal{#1}{text}}{\let\RightBracketSize=\NextTextStyle}{
        \ifthenelse{\equal{#1}{d}\OR\equal{#1}{display}}{\let\RightBracketSize=\NextDisplayStyle}{
          \ifthenelse{\equal{#1}{a}\OR\equal{#1}{auto}}{\let\RightBracketSize=\right}{
            \let\RightBracketSize=\relax}}}}}}}}}}

\title{Variational Inequalities and Improved Convergence Rates for Tikhonov Regularisation on Banach Spaces}

\author{Markus Grasmair\\
\\
\normalsize
\begin{tabular}{c}
  Computational Science Center\\
  University of Vienna\\
  Nordbergstr.~15\\
  A--1090 Vienna, Austria
\end{tabular}
}

\date{July 13, 2011}

\begin{document}

\maketitle

\begin{abstract}
  In this paper we derive higher order convergence rates in terms of the Bregman distance
  for Tikhonov like convex regularisation for linear operator equations
  on Banach spaces.
  The approach is based on the idea of variational inequalities,
  which are, however, not imposed on the original Tikhonov functional,
  but rather on a dual functional.
  Because of that, the approach is not limited to convergence
  rates of lower order, but yields the same range of rates
  that is well known for quadratic regularisation on Hilbert spaces.
\end{abstract}

\section{Introduction}

Classical results in the theory of quadratic Tikhonov regularisation on Hilbert spaces show
that the asymptotic quality of the regularised solution of 
an ill-posed linear operator equation $Ax = y$
depends strongly on the smoothness of the true solution $x^\dagger$
with respect to the operator $A\colon X \to Y$:
If $x^\dagger$ satisfies the range condition 
\begin{equation}\label{eq:nu}
  x^\dagger\in\range(A^*A)^\nu
\end{equation}
for some $0 < \nu \le 1$ and the accuracy of the
right hand side of the equation is of order $\delta$,
then one obtains, with a suitable choice of
the regularisation parameter, an accuracy of the regularised solution
of order $\delta^{2\nu/(2\nu+1)}$.
More precisely, denote by $y^\delta\in Y$ any right hand side satisfying
$\norm{y-y^\delta} \le \delta$,
and by $x_\alpha^\delta\in X$ with $\alpha > 0$ and $\delta > 0$ any element
satisfying
\[
x_\alpha^\delta = \argmin\set{\norm{Ax-y^\delta}^2 + \alpha\norm{x}^2}{x \in X}
\]
for some $y^\delta$.
Then a choice of the regularisation parameter $\alpha$
of order $\alpha \sim \delta^{2/(2\nu+1)}$ implies that
$\norm{x_\alpha^\delta-x^\dagger} = O(\delta^{2\nu/(2\nu+1)})$
(see for instance~\cite{EngHanNeu96} or~\cite{Gro84}).

The situation becomes more complicated in the case where
the regularisation term is not the squared Hilbert space norm
and, even worse, the regularisation takes place only in a Banach space
and not a Hilbert space setting.
Then Tikhonov regularisation consists in the minimisation of a functional
\[
\mathcal{T}_\alpha(x;y^\delta) = \norm{Ax-y^\delta}^2 + \alpha\mathcal{R}(x)\,,
\]
where $\mathcal{R}$ is some convex and lower semi-continuous functional
on $X$.
The first results that treated convergence rates for
non-quadratic regularisation remained close to the quadratic case.
In~\cite{Neu87} and~\cite{ChaKun94a}, quadratic regularisation
with convexity constraints was discussed, leading to conditions
and results that are very similar to the unconstrained case.
The first result that transcended the Hilbert space setting
was the treatise~\cite{EngLan93}, where convergence rates
for maximum entropy regularisation were derived.
The main argument, however, was a translation of the non-quadratic
problem to an equivalent quadratic problem on a Hilbert space.

One of the main difficulty in the generalisation of the convergence rates
result to general convex regularisation methods is that the norm
on $X$, in which the convergence rate was usually measured, need not be
related to the regularisation term that should imply the convergence.
For this reason, it was argued in~\cite{BurOsh04}
that it is more natural to measure the quality of the approximation not
in terms of the norm but rather in terms of the Bregman distance
with respect to the regularisation term $\mathcal{R}$.
Then it is easy to derive, for a parameter choice $\alpha\sim\delta$,
a convergence rate of order $\delta$ provided the range condition 
$\range A^* \cap \partial\mathcal{R}(x^\dagger)\neq \emptyset$ holds.
Here $\partial\mathcal{R}(x^\dagger)$ denotes the sub-differential
of the convex functional $\mathcal{R}$ at $x^\dagger$.
This result corresponds to the case $\nu = 1/2$ in~\eqref{eq:nu}.
In addition, a result with the improved range condition
$\range A^*A\cap\partial\mathcal{R}(x^\dagger) \neq \emptyset$
with $Y$ being a Hilbert space has been derived in~\cite{Res05}
(see also~\cite{Neu09,NeuHeiHofKinTau10} for more general results).

The above results on convergence rates for regularisation
on Banach spaces correspond to the cases $\nu = 1/2$ or $\nu = 1$
in~\eqref{eq:nu}; none of the intermediate cases was treated in~\cite{BurOsh04} or~\cite{Res05}.
There are two main reasons:
First, an intermediate range condition is only possible
if fractional powers of $A$ can be defined, which basically
limits the theory to Hilbert spaces, and, second,
the classical proofs of the intermediate convergence rates
rely on the possibility of writing $x_\alpha^\delta$ as the solution
of a linear operator equation.

An alternative approach for the derivation of convergence rates
in a quadratic Hilbert space setting was introduced in~\cite{Hof06}
(see also~\cite{Sch01a}, where a similar idea was used in a slightly different context)
under the headings of \emph{approximate source conditions}
and \emph{distance functions}.
This approach was used with success in~\cite{Hei09}
(note also the preliminary results in~\cite{Hei08b}),
leading to the whole range of estimates also available for
the quadratic case.
Another alternative, \emph{variational inequalities},
arised in~\cite{HofKalPoeSch07} (see also~\cite[Sec.~3.2]{SchGraGroHalLen09})
from the search for a natural way of extending convergence rates results
to non-linear problems.
In~\cite{BotHof10,Gra10b} it was shown that these variational inequalities
can be used for covering the whole range of \emph{slow} convergence
rates corresponding to the cases $0 < \nu \le 1/2$ in~\eqref{eq:nu}.
The two concepts were also combined in~\cite{FleHof10}
as \emph{approximate variational inequalities},
but the obtained rates were not different from those in~\cite{BotHof10,Gra10b}.

In~\cite{Fle11}, the different new concepts for obtaining (slow) convergence rates
for convex regularisation methods in Banach spaces were discussed
and compared in the case of quadratic regularisation on Hilbert spaces.
It turned out that, in this simple setting, they were equivalent.
In more complicated situations, however, in particular for the regularisation
of non-linear operators, but also for non-convex regularisation (see~\cite{Gra10b}),
variational inequalities might be easier to handle.
It is, however, necessary to show that they can also be used
for deriving fast convergence rates.

In this paper, it is shown that this is indeed possible 
for linear operators on Banach spaces.
The proof and also the formulation of the variational inequality
are based on a duality argument that shows that,
if the standard range condition $\range A^* \cap \partial\mathcal{R}(x^\dagger)$
holds, then the minimisation of the Tikhonov functional $\mathcal{T}_\alpha$
with noisy data $y^\delta$ can be translated to the approximative minimisation
of a dual Tikhonov functional with exact data.
In the special case of a sufficiently smooth regularisation term $\mathcal{R}$
with sufficiently smooth dual $\mathcal{R}^*$,
the ensuing condition for the best possible rates with this method
is almost identical to the range condition for the Hilbert space
corresponding to the case $\nu = 1$ in~\eqref{eq:nu}.

\section{Notation}
                                   
Let $X$ and $Y$ be Banach spaces and $A\colon X\to Y$ a
bounded linear operator.
Moreover assume that $\mathcal{R}\colon X \to [0,+\infty]$
is a convex, lower semi-continuous, and coercive regularisation functional and $p > 1$.
Consider, for given data $y \in Y$ and a regularisation
parameter $\alpha > 0$, the Tikhonov functional
$\mathcal{T}(\cdot;\alpha,y) \colon X \to [0,+\infty]$ defined by
\[
\mathcal{T}_\alpha(x;y) := \frac{1}{p}\norm{Ax-y}^p + \alpha \mathcal{R}(x)\;.
\]

From now on we assume that $y^\dagger \in Y$ are some fixed data,
and we want to solve the equation
\[
Ax = y^\dagger
\]
for $x$.
By
\[
x^\dagger :\in \argmin \set{\mathcal{R}(x)}{Ax = y^\dagger}
\]
we denote any $\mathcal{R}$-minimising solution of the equation $Ax = y^\dagger$.
Moreover, for $\alpha > 0$ we denote by
\[
x_\alpha :\in \argmin\set{\mathcal{T}_\alpha(x;y^\dagger)}{x \in X}
\]
any minimiser of the Tikhonov functional with exact data $y^\dagger$.
Finally, for $\delta > 0$ and $\alpha > 0$ we denote by
\[
x_\alpha^\delta := \in\argmin\set{\mathcal{T}_\alpha(x;y^\delta)}{x \in X}\,,
\]
any minimiser of the Tikhonov functional with perturbed data $y^\delta$,
where $y^\delta \in Y$ is any element satisfying
\[
\norm{y^\delta - y^\dagger} \le \delta\;.
\]
The main goal of this paper is the derivation of \emph{convergence rates},
that is, estimates of the distance between $x_\alpha^\delta$ and $x^\dagger$
depending on the noise level $\delta$ and the regularisation parameter $\alpha$.
Starting with~\cite{BurOsh04}, such estimates have been usually derived in terms
of the Bregman distance defined by the convex regularisation functional $\mathcal{R}$.

\begin{definition}
  Assume that $x \in X$ and $\xi \in \partial\mathcal{R}(x)$.
  Then we define the \emph{Bregman distance} with respect to $\mathcal{R}$
  as the mapping $\mathcal{D}_\xi(\cdot;x) \colon X \to [0,+\infty]$,
  \[
  \mathcal{D}_\xi(\tilde{x};x) = \mathcal{R}(\tilde{x})-\mathcal{R}(x)-\inner{\xi}{\tilde{x}-x}\;.
  \]
  That is, the Bregman distance measures the difference between
  the graph of $\mathcal{R}$ and its affine approximation at $x$.

  In addition, if $x$, $\tilde{x} \in X$ and $\xi \in \partial\mathcal{R}(x)$, 
  $\tilde{\xi}\in\partial\mathcal{R}(\tilde{x})$, then we denote by
  \[
  \mathcal{D}_{\xi,\tilde{\xi}}^{\rm sym}(x,\tilde{x})
  := \mathcal{D}_\xi(\tilde{x};x) + \mathcal{D}_{\tilde{\xi}}(x;\tilde{x})
  \]
  the \emph{symmetrical Bregman distance} between $x$ and $\tilde{x}$
  with respect to $\mathcal{R}$.
  An easy computation shows that the identity
  \[
  \mathcal{D}_{\xi,\tilde{\xi}}^{\rm sym}(x,\tilde{x})
  = \inner{\xi-\tilde{\xi}}{x-\tilde{x}}
  \]
  holds.
\end{definition}

Note that the Bregman distance at $x$ depends on the choice
of the sugradient $\xi \in \partial\mathcal{R}(x)$ 
(unless $\partial\mathcal{R}(x)$ contains a single element).
Moreover it can happen that $\partial\mathcal{R}(x) = \emptyset$,
in which case the Bregman distance cannot be defined.

\section{Duality Mappings}

In addition to the Bregman distance with respect to $\mathcal{R}$,
we will also need the Bregman distance with respect
to the $p$-th power of the norm on the space $Y$.
This Bregman distance can be written in terms of duality mappings.
In this section we recall several results concerning duality mappings
that will be needed for the derivation and interpretation 
of the main result on convergence rates.

Throughout the whole text, whenever $q > 1$ we denote by $q_*$ its
conjugate defined by $1/q + 1/q_* = 1$.

\begin{definition}
  Let $q > 1$. The \emph{duality mapping} $J_q \colon Y \to 2^{Y^*}$
  is defined by
  \[
  J_q(y) = \set{\omega\in Y^*}{\inner{\omega}{y} = \norm{\omega}\norm{y},\ \norm{\omega} = \norm{y}^{q-1}}\;.
  \]
  The \emph{adjoint duality mapping}
  $J_{q_*}^* \colon Y^* \to 2^Y$ is defined by
  \[
  J_{q_*}^*(\omega) = \set{y\in Y}{\inner{\omega}{y} = \norm{\omega}\norm{y},\ \norm{y} = \norm{\omega}^{q_*-1}}\;.
  \]
\end{definition}

If $Y$ is reflexive, one can show that $J_{q_*}^* = (J_q)^{-1}$
(see~\cite[Chap.~1, Thm.~4.4]{Cio90}).
Moreover it is easy to see that $J_q$ is $(q-1)$-homogeneous.
That is, if $y \in Y$ and $\lambda \in \R$, then
\[
J_q(\lambda y) = \abs{\lambda}^{q-1}J_q(y)\;.
\]

\begin{lemma}
  Let $q > 1$ and denote by $\mathcal{S}_q\colon Y \to [0,+\infty]$ the mapping
  \[
  \mathcal{S}_q(y) = \frac{1}{q}\norm{y}^q\;.
  \]
  Then
  \[
  J_q(y) = \partial\mathcal{S}_q(y)\;.
  \]
\end{lemma}

\begin{proof}
  See \cite[Chap.~2, Cor.~3.5]{Cio90}.
\end{proof}

Recall that the Banach space $Y$ is \emph{smooth}
(see~\cite{LinTza79}),
if for every $y \in Y$ satisfying $\norm{y} = 1$ there exists a unique $\omega \in Y^*$
such that $\norm{\omega} = \inner{\omega}{y} = 1$.
Equivalently, $Y$ is smooth, if and only if every duality
mapping $J_q$ on $Y$ with $q > 1$ is single valued.
In this case, we regard $J_q$ as a mapping from $Y$ to $Y^*$.
This allows us to formulate the Bregman distance with respect to $\mathcal{S}_q$
in terms of the duality mapping $J_q$.

\begin{definition}
  Assume that $Y$ is a smooth Banach space and let $q > 1$.
  For $y \in Y$ we define the Bregman distance with respect to $\mathcal{S}_q$
  at $y$ by
  \[
  \mathcal{D}_q(\tilde{y};y)
  := \frac{1}{q}\norm{\tilde{y}}^q - \frac{1}{q}\norm{y}^q - \inner{J_q(y)}{\tilde{y}-y}\;.
  \]
  Moreover we define the symmetric Bregman distance with respect to $\mathcal{S}_q$ as
  \[
  \mathcal{D}_q^{\rm sym}(\tilde{y},y)
  := \inner{J_q(\tilde{y})-J_q(y)}{\tilde{y}-y}\;.
  \]  
\end{definition}

The following definition is taken from~\cite{LinTza79}.

\begin{definition}
  The modulus of convexity $\delta_Y \colon (0,2] \to [0,1]$ is defined by
  \[
  \delta_Y(\eps) := \inf\set{1-\norm{y-\tilde{y}}/2}{y,\,\tilde{y} \in Y,\ \norm{y}=\norm{\tilde{y}}=1,\ \norm{y-\tilde{y}}=\eps}\;.
  \]
  The modulus of smoothness $\rho_Y\colon (0,+\infty) \to (0,+\infty)$ is defined by
  \[
  \rho_Y(\tau) = \sup\set{(\norm{y+\tilde{y}} + \norm{y-\tilde{y}})/2-1}{y,\,\tilde{y} \in Y,\ \norm{y}=1,\ \norm{\tilde{y}}=\tau}\;.
  \]
  The space $Y$ is called \emph{uniformly convex} if $\delta_Y(\eps) > 0$ for every $\eps > 0$.
  It is called \emph{uniformly smooth} if $\rho_Y(\tau) = o(\tau)$ as $\tau \to 0$.

  The space $Y$ is called \emph{$q$-convex} (or convex of power type $q$), if 
  there exists a constant $K > 0$ such that $\delta_Y(\eps) \ge K\eps^q$ for all $\eps$.
  Similarly, it is called \emph{$q$-smooth} (or smooth of power type $q$), if
  $\rho_Y(\tau) \le K\tau^q$ for all sufficiently small $\tau > 0$.
\end{definition}

One can show that the space $Y$ is $q$-convex, if and only if $Y^*$ is $q_*$-smooth
(in fact, the relation $2\rho_{Y^*}=(2\delta_Y)^*$ holds, see~\cite[Prop.~1.e.2]{LinTza79}).

For us, the most important property of $q$-convex spaces
is that they allow to estimate the $q$-th power of the norm from above
by the Bregman distance.
More precisely, the following result holds:

\begin{lemma}\label{le:qconvex}
  Assume that the Banach space $Y$ is $q$-convex.
  Then there exists a constant $C > 0$ such that
  \[
  C\norm{y-\tilde{y}}^q \le \mathcal{D}_q(\tilde{y};y)\;.
  \]
  for all $y$, $\tilde{y} \in Y$.
\end{lemma}

\begin{proof}
  See~\cite[Lemma~2.7]{BonKazMaaSchoSchu08}.
\end{proof}

\section{Convergence Rates}

In the following we always assume that $x^\dagger$ satisfies a range condition with
\[
\xi^\dagger = A^*\omega^\dagger \in \range A^* \cap \partial\mathcal{R}(x^\dagger)\;.
\]
Then the definition of the dual $\mathcal{R}^*\colon X^* \to (-\infty,+\infty]$ of $\mathcal{R}$ implies that
\[
x^\dagger \in \partial\mathcal{R}^*(A^*\omega^\dagger)\;.
\]
Consequently we can define a Bregman distance 
$\mathcal{D}^*_{x^\dagger}(\cdot;A^*\omega^\dagger) \colon X^* \to [0,+\infty]$
for the dual function $\mathcal{R}^*$ as
\[
\mathcal{D}^*_{x^\dagger}(\xi;A^*\omega^\dagger) = \mathcal{R}^*(\xi) - \mathcal{R}^*(A^*\omega^\dagger) - \inner{\xi-A^*\omega^\dagger}{x^\dagger}\;.
\]

\begin{proposition}\label{pr:dual_Tik}
  Assume that $x_\alpha^\delta$ minimises $\mathcal{T}_\alpha(\cdot;y^\delta)$.
  Then there exists $\omega_\alpha^\delta \in Y^*$ such that
  \begin{equation}\label{eq:KKT}
    A^*\omega_\alpha^\delta \in \partial\mathcal{R}(x_\alpha^\delta)\,,
    \qquad\qquad
    -\alpha\omega_\alpha^\delta = J_p(Ax_\alpha^\delta - y^\delta)\;.
  \end{equation}
  Moreover $\omega_\alpha^\delta$ minimises the functional
  $\mathcal{T}^*_\alpha(\cdot;y^\delta) \colon Y^* \to (-\infty,+\infty]$ defined by
  \[
  \mathcal{T}^*_\alpha(\omega;y^\delta) 
  := \mathcal{D}^*_{x^\dagger}(A^*\omega;A^*\omega^\dagger) + \alpha^{p_*-1}\frac{1}{p_*}\norm{\omega}^{p_*}
  -\inner{\omega-\omega^\dagger}{y^\delta-y^\dagger}\;.
  \]
\end{proposition}

\begin{proof}
  The dual of the problem 
  \[
  \mathcal{T}_\alpha(x)/\alpha = \frac{1}{\alpha p}\norm{Ax-y^\delta}^p + \mathcal{R}(x) \to \min
  \]
  is the problem 
  \[
  \frac{1}{\alpha p_*}\norm{\alpha\omega}^{p_*} - \inner{\omega}{y^\delta} + \mathcal{R}^*(A^*\omega) \to \min
  \]
  (see~\cite[Chap.~III]{EkeTem74}).
  Writing
  \[
  -\inner{\omega}{y^\delta} = \inner{\omega}{y^\dagger-y^\delta}-\inner{\omega}{y^\dagger}
  = -\inner{\omega}{y^\delta-y^\dagger} - \inner{A^*\omega}{x^\dagger}
  \]
  and adding the
  constant terms $\inner{A^*\omega^\dagger}{x^\dagger}$,
  $\inner{\omega^\dagger}{y^\delta-y^\dagger}$,
  and $-\mathcal{R}^*(A^*\omega^\dagger)$,
  we obtain the problem
  \begin{multline*}
  \alpha^{p_*-1}\frac{1}{p_*}\norm{\omega}^{p_*} + \mathcal{R}^*(A^*\omega) 
  - \mathcal{R}^*(A^*\omega^\dagger)\\
  - \inner{A^*\omega-A^*\omega^\dagger}{x^\dagger}
  - \inner{\omega-\omega^\dagger}{y^\delta-y^\dagger}
  \to \min\,,
  \end{multline*}
  which is precisely the minimisation of the dual Tikhonov functional $\mathcal{T}^*_\alpha$.
  The relations~\eqref{eq:KKT} are nothing else than
  the Karush--Kuhn--Tucker conditions for the minimisation problem
  (cf.~\cite[Chap.~III]{EkeTem74}).
\end{proof}

\begin{remark}
  In particular, Proposition~\ref{pr:dual_Tik} implies that,
  if $x_\alpha^\delta$ minimises the primal Tikhonov functional with noisy
  data $y^\delta$, then $\omega_\alpha^\delta$ \emph{almost} minimises
  the dual Tikhonov functional $\mathcal{T}_\alpha^*(\cdot;y^\dagger)$
  with \emph{exact} data $y^\dagger$.
  More precisely, we have
  \[
  \mathcal{T}_\alpha^*(\omega_\alpha^\delta;y^\dagger)
  \le \inf\set{\mathcal{T}_\alpha^*(\omega;y^\dagger)}{\omega \in Y^*} + \delta\norm{\omega_\alpha^\delta-\omega^\dagger}\;.
  \]
\end{remark}

\begin{definition}
  An \emph{index function} is a strictly increasing, continuous
  and concave function $\Phi\colon [0,+\infty) \to [0,+\infty)$
  satisfying $\Phi(0) = 0$.
\end{definition}

\begin{theorem}\label{th:main}
  Assume that $Y$ is a $p$-smooth Banach space
  and that there exists an index function $\Phi$ such that
  \begin{equation}\label{eq:var_ineq}
  \inner{\omega-\omega^\dagger}{J_{p_*}(\omega^\dagger)}
  \le \Phi\bigl(\mathcal{D}_{x^\dagger}^*(A^*\omega;A^*\omega^\dagger)\bigr)
  \end{equation}
  for every $\omega \in Y^*$.
  Denote by $\Psi$ the conjugate of the convex mapping $t \mapsto \Phi^{-1}(t)$.
  Then there exists a constant $D$ only depending on $p$ and the Banach space $Y$ such that
  \[
  \mathcal{D}_{A^*\omega^\dagger}(x_\alpha^\delta;x^\dagger) \le \Psi(\alpha^{p_*-1}) + D\frac{\delta^p}{\alpha}\;.
  \]
\end{theorem}

\begin{proof}
  We have
  \[
  \begin{aligned}
    \mathcal{D}_{A^*\omega_\alpha^\delta,A^*\omega^\dagger}^{\rm sym}(x^\dagger,x_\alpha^\delta)
    &= \inner{A^*\omega_\alpha^\delta-A^*\omega^\dagger}{x_\alpha^\delta-x^\dagger}\\
    &= \inner{\omega_\alpha^\delta-\omega^\dagger}{Ax_\alpha^\delta-Ax^\dagger}\\
    &= \inner{\omega_\alpha^\delta-\omega^\dagger}{Ax_\alpha^\delta-y^\delta}
       +\inner{\omega_\alpha^\delta-\omega^\dagger}{y^\delta-y^\dagger}\;.
  \end{aligned}
  \]
  Proposition~\ref{pr:dual_Tik} implies that
  $-\alpha\omega_\alpha^\delta = J_p(Ax_\alpha^\delta-y^\delta)$,
  and therefore
  \[
  Ax_\alpha^\delta-y^\delta 
  = J_{p_*}(-\alpha\omega_\alpha^\delta)
  = -\alpha^{p_*-1}J_{p_*}(\omega_\alpha^\delta)\;.
  \]
  Consequently
  \[
  \begin{aligned}
    \mathcal{D}_{A^*\omega_\alpha^\delta,A^*\omega^\dagger}^{\rm sym}(x^\dagger,x_\alpha^\delta)
    &= -\alpha^{p_*-1}\inner{\omega_\alpha^\delta-\omega^\dagger}{J_{p_*}(\omega_\alpha^\delta)}
       +\inner{\omega_\alpha^\delta-\omega^\dagger}{y^\delta-y^\dagger}\\
    &= -\alpha^{p_*-1}\inner{\omega_\alpha^\delta-\omega^\dagger}{J_{p_*}(\omega^\dagger)}
       -\alpha^{p_*-1}\mathcal{D}_{p_*}^{\rm sym}(\omega_\alpha^\delta,\omega^\dagger)\\
    &\qquad{} +\inner{\omega_\alpha^\delta-\omega^\dagger}{y^\delta-y^\dagger}
  \end{aligned}
  \]
  Using~\eqref{eq:var_ineq} and estimating 
  $\inner{\omega_\alpha^\delta-\omega^\dagger}{y^\delta-y^\dagger} \le \delta\norm{\omega_\alpha^\delta-\omega^\dagger}$,
  we obtain
  \begin{multline}\label{eq:main1}
    \mathcal{D}_{A^*\omega_\alpha^\delta,A^*\omega^\dagger}^{\rm sym}(x^\dagger,x_\alpha^\delta)\\
    \le \alpha^{p_*-1}\Phi\bigl(\mathcal{D}_{x^\dagger}^*(A^*\omega_\alpha^\delta;A^*\omega^\dagger)\bigr)
       -\alpha^{p_*-1}\mathcal{D}_{p_*}^{\rm sym}(\omega_\alpha^\delta,\omega^\dagger)
       + \delta\norm{\omega_\alpha^\delta-\omega^\dagger}\;.
  \end{multline}
  Now,
  \[
  \mathcal{D}_{A^*\omega_\alpha^\delta,A^*\omega^\dagger}^{\rm sym}(x^\dagger,x_\alpha^\delta)
  = \mathcal{D}_{A^*\omega^\dagger}(x_\alpha^\delta;x^\dagger) + \mathcal{D}_{A^*\omega_\alpha^\delta}(x^\dagger;x_\alpha^\delta)\,,
  \]
  and the assumptions $A^*\omega^\dagger\in\partial\mathcal{R}(x^\dagger)$
  and $A^*\omega_\alpha^\delta \in \partial\mathcal{R}(x_\alpha^\delta)$
  imply that
  \[
  \begin{aligned}
    \mathcal{D}_{A^*\omega_\alpha^\delta}(x^\dagger;x_\alpha^\delta)
    &= \mathcal{R}(x^\dagger)-\mathcal{R}(x_\alpha^\delta) - \inner{A^*\omega_\alpha^\delta}{x^\dagger-x_\alpha^\delta}\\
    &= \inner{A^*\omega^\dagger}{x^\dagger}-\mathcal{R}^*(A^*\omega^\dagger)
    - \inner{A^*\omega_\alpha^\delta}{x_\alpha^\delta} + \mathcal{R}^*(A^*\omega_\alpha^\delta)\\
    &\qquad {}-\inner{A^*\omega_\alpha^\delta}{x^\dagger-x_\alpha^\delta}\\
    &= \mathcal{R}^*(A^*\omega_\alpha^\delta)-\mathcal{R}^*(\omega^\dagger) - \inner{A^*(\omega_\alpha^\delta-\omega^\dagger)}{x^\dagger}\\
    &=  \mathcal{D}^*_{x^\dagger}(A^*\omega_\alpha^\delta;A^*\omega^\dagger)\;.
  \end{aligned}
  \]
  Because $Y$ is $p$-smooth, its dual space $Y^*$ is $p_*$-convex.
  Thus there exists $C > 0$ such that (see Lemma~\ref{le:qconvex})
  \[
  C\norm{\omega_\alpha^\delta-\omega^\dagger}^{p_*} \le \mathcal{D}_{p_*}^{\rm sym}(\omega_\alpha^\delta,\omega^\dagger)\;.
  \]
  Consequently~\eqref{eq:main1} implies that
  \begin{multline*}
  \mathcal{D}_{A^*\omega^\dagger}(x_\alpha^\delta;x^\dagger)
  \le \alpha^{p_*-1}\Phi\bigl(\mathcal{D}_{x^\dagger}^*(A^*\omega_\alpha^\delta;A^*\omega^\dagger)\bigr)
  - \mathcal{D}_{x^\dagger}^*(A^*\omega_\alpha^\delta;A^*\omega^\dagger)\\
  + \delta\norm{\omega_\alpha^\delta-\omega^\dagger} - C\alpha^{p_*-1}\norm{\omega_\alpha^\delta-\omega^\dagger}^{p_*}\;.
  \end{multline*}
  Applying Young's inequality (see~\cite[Thm.~13.2]{HewStr65}), we obtain
  \[
  \begin{aligned}
  \alpha^{p_*-1}\Phi\bigl(\mathcal{D}_{x^\dagger}^*(A^*\omega_\alpha^\delta;A^*\omega^\dagger)\bigr)
  &\le \Psi(\alpha^{p_*-1}) + \Psi^*\bigl(\Phi\bigl(\mathcal{D}_{x^\dagger}^*(A^*\omega_\alpha^\delta;A^*\omega^\dagger)\bigr)\bigr)\\
  &= \Psi(\alpha^{p_*-1}) + \mathcal{D}_{x^\dagger}^*(A^*\omega_\alpha^\delta;A^*\omega^\dagger)
  \end{aligned}
  \]
  and
  \[
  \delta\norm{\omega_\alpha^\delta-\omega^\dagger}
  \le C\alpha^{p_*-1}\norm{\omega_\alpha^\delta-\omega^\dagger}^{p_*}
  + \frac{1}{pp_*^{1/p_*}C^{p/p_*}} \frac{\delta^p}{\alpha}\;.
  \]
  Therefore
  \[
  \mathcal{D}_{A^*\omega^\dagger}(x_\alpha^\delta;x^\dagger) \le \Psi(\alpha^{p_*-1}) + \frac{1}{pp_*^{1/p_*}C^{p/p_*}} \frac{\delta^p}{\alpha}\,,
  \]
  which proves the assertion.
\end{proof}

\begin{corollary}
  Assume that the assumptions of Theorem~\ref{th:main} are satisfied.
  Then we have for \emph{exact data} $y^\dagger$ the estimate
  \[
  \mathcal{D}_{A^*\omega^\dagger}(x_\alpha;x^\dagger) \le \Psi(\alpha^{p_*-1})\;.
  \]
  In the special case where $\Phi(t) = c t^{1/q_*}$ for some $q_* > 1$ and $c > 0$ we have
  \[
  \mathcal{D}_{A^*\omega^\dagger}(x_\alpha;x^\dagger) \le \frac{c^{q}q_*^{1-q}}{q} \alpha^{(p_*-1)q}\;.
  \]
\end{corollary}

\begin{proof}
  In the general case, the assertion directly follows from Theorem~\ref{th:main}
  for $\delta = 0$.
  In the case $\Phi(t) = ct^{1/q_*}$, the mapping $\Psi$ is the conjugate
  of $\Phi^{-1}(t) = t^{q_*}/c^{q_*}$.
  A short calculation shows that $\Psi$ indeed has the form stated above.
\end{proof}

\begin{corollary}
  Assume that the assumptions of Theorem~\ref{th:main} are satisfied with
  an index function $\Phi(t) \sim t^{1/q_*}$ for some $q_* > 1$.
  Then we have for a parameter choice 
  \[
  \alpha \sim \delta^{\frac{p}{(p_*-1)q+1}}
  \]
  the convergence rate
  \[
  \mathcal{D}_{A^*\omega^\dagger}(x_\alpha^\delta;x^\dagger)
  = O\bigl(\delta^{\frac{p_*q}{(p_*-1)q+1}}\bigr)\;.
  \]
\end{corollary}

\begin{proof}
  From Theorem~\ref{th:main} we obtain the estimate
  \[
  \mathcal{D}_{A^*\omega^\dagger}(x_\alpha^\delta;x^\dagger) 
  \le \Psi(\alpha^{p_*-1}) + D\frac{\delta^p}{\alpha}\,,
  \]
  where $\Psi$ is the conjugate of the function $\Phi^{-1}$.
  Since $\Phi(t) \sim t^{1/q_*}$, it follows that $\Phi^{-1} \sim t^{q_*}$,
  and therefore $\Psi \sim t^q$.
  Thus
  \[
  \Psi(\alpha^{p_*-1}) + D \frac{\delta^p}{\alpha}
  \sim \alpha^{(p_*-1)q} + \frac{\delta^p}{\alpha}
  \sim \delta^{\frac{p(p_*-1)q}{(p_*-1)q+1}}+\delta^{p-\frac{p}{(p_*-1)q+1}}
  \sim \delta^{\frac{p(p_*-1)q}{(p_*-1)q+1}}\,,
  \]
  which proves the assertion,
  as $p(p_*-1) = p_*$.
\end{proof}

\section{Implications of Smoothness}

In this section, we first discuss the limitations of Theorem~\ref{th:main}
in the case where the dual of $\mathcal{R}$ is smooth.
Then the best possible rates imply that the dual variable
$\omega^\dagger$ satisfies the range condition $J_{p_*}(\omega^\dagger) \in \range A$.
Conversely, this range condition implies rates, if the functional
$\mathcal{R}$ itself is sufficiently smooth.

\begin{lemma}
  Assume that $\mathcal{R}^*$ is two times Fr\'echet differentiable at $A^*\omega^\dagger$
  and that there exists an index function $\Phi$ such that
  \begin{equation}\label{eq:smooth}
    \inner[b]{\omega-\omega^\dagger}{J_{p_*}(\omega^\dagger)}\le\Phi\bigl(\mathcal{D}_{x^\dagger}^*(A^*\omega;A^*\omega^\dagger)\bigr)
  \end{equation}
  for all $\omega \in Y^*$.
  If $\Phi(t) = o(t^{1/2})$ as $t \to 0$, then $x^\dagger$ minimises $\mathcal{R}$.

  In addition, if one can choose $\Phi(t) \sim t^{1/2}$ as $t \to 0$, then
  \[
  J_{p_*}(\omega^\dagger) \in \range A\;.
  \]
\end{lemma}

\begin{proof}
  Because $\mathcal{R}^*$ is two times Fr\'echet differentiable at $A^*\omega^\dagger$,
  a Taylor expansion of $\mathcal{R}^*$ around $A^*\omega^\dagger$ shows that
  \begin{multline*}
  \mathcal{D}^*_{x^\dagger}(A^*\omega;A^*\omega^\dagger)
  = \mathcal{R}^*(A^*\omega) - \mathcal{R^*}(A^*\omega^\dagger) - \inner[b]{x^\dagger}{A^*(\omega-\omega^\dagger)}\\
  = (\mathcal{R}^*)''(A^*\omega^\dagger)\bigl(A^*(\omega-\omega^\dagger);A^*(\omega-\omega^\dagger)\bigr) + o(\norm{A^*(\omega-\omega^\dagger)}^2)\,,
  \end{multline*}
  where $(\mathcal{R}^*)''(A^*\omega^\dagger) \in B(X^*)$ denotes the second order
  derivative of $\mathcal{R}^*$ at $A^*\omega^\dagger$, which is a symmetric, bounded bilinear form on $X^*$.
  Writing $\tilde{\omega} := \omega-\omega^\dagger$,
  the inequality~\eqref{eq:smooth} implies that
  \[
  \inner[b]{\tilde{\omega}}{J_{p_*}(\omega^\dagger)}
  \le \Phi\bigl((\mathcal{R}^*)''(A^*\omega^\dagger)(A^*\tilde{\omega};A^*\tilde{\omega})+o(\norm{A^*\tilde{\omega}}^2)\bigr)
  \]
  for all $\tilde{\omega} \in X^*$.
  In addition, one can estimate
  \[
  (\mathcal{R}^*)''(A^*\tilde{\omega};A^*\tilde{\omega}) \le \norm{(\mathcal{R}^*)''(A^*\omega^\dagger)} \norm{A^*\tilde{\omega}}^2\;.
  \]
  Therefore,
  \begin{equation}\label{eq:smooth1}
    \inner{\tilde{\omega}}{J_{p_*}(\omega^\dagger)}
    \le \Phi\bigl(\norm{(\mathcal{R}^*)''(A^*\omega^\dagger)} \norm{A^*\tilde{\omega}}^2+o(\norm{A^*\tilde{\omega}}^2)\bigr)\;.
  \end{equation}
  
  Now assume that $\Phi(t) = o(t^{1/2})$ as $t \to 0$.
  Then, dividing~\eqref{eq:smooth1} by $\norm{A^*\tilde{\omega}}$ and considering
  the limit $\norm{A^*\tilde{\omega}} \to 0$, we see that $J_{p_*}(\omega^\dagger) = 0$,
  which is equivalent to stating that $\omega^\dagger = 0$.
  Thus $0 = A^*\omega^\dagger \in \partial\mathcal{R}(x^\dagger)$,
  which proves that $x^\dagger$ minimises $\mathcal{R}$.

  Now assume that $\Phi(t) \sim t^{1/2}$ as $t \to 0$.
  Then~\eqref{eq:smooth1} implies that there exists a constant $C > 0$ such that
  \[
  \inner[b]{\tilde{\omega}}{J_{p_*}(\omega^\dagger)}
  \le C \bigl(\norm{(\mathcal{R}^*)''(A^*\omega^\dagger)}\bigr)^{1/2}\norm{A^*\tilde{\omega}}
  \]
  for $\norm{A^*\tilde{\omega}}$ sufficiently small---and thus everywhere,
  as the right hand side is positively homogeneous and the left hand side is linear.
  Thus~\cite[Lemma~8.21]{SchGraGroHalLen09} implies that $J_{p_*}(\omega^\dagger) \in \range A$.
\end{proof}

In the special case, where $X$ is a $q$-convex Banach space and
\[
\mathcal{R}(x) = \frac{1}{q}\norm{x}^q\,,
\]
the above result can be slightly refined.
In this case, the condition of Theorem~\ref{th:main} reads as
\[
\inner[b]{\omega-\omega^\dagger}{J_{p_*}(\omega^\dagger)} \le \Phi\bigl(\mathcal{D}_{q_*}(A^*\omega;A^*\omega^\dagger)\bigr)
\]
and the resulting estimate is
\[
\mathcal{D}_q(x_\alpha^\delta;x^\dagger) \le \Psi(\alpha^{p_*-1}) + D\frac{\delta^p}{\alpha}\;.
\]

\begin{lemma}\label{le:qlimit}
  Assume that $X$ is $q$-convex and
  \[
  \mathcal{R}(x) = \frac{1}{q}\norm{x}^q\,,
  \]
  and that
  \[
  \inner[b]{\omega-\omega^\dagger}{J_{p_*}(\omega^\dagger)}
  \le \Phi\bigl(\mathcal{D}_{q_*}(A^*\omega;A^*\omega^\dagger)\bigr)
  \]
  for some index function $\Phi$.
  If
  \[
  \Phi(t) = o(t^{1/{q_*}})
  \qquad\qquad
  \text{ as } q \to 0\,,
  \]
  then $x^\dagger = 0$.
\end{lemma}

\begin{proof}
  We have
  \[
  \Phi\bigl(\mathcal{D}_{q_*}(A^*\omega;A^*\omega^\dagger)\bigr)
  \le \Phi\bigl(\hat{C}_{q_*}\norm{A^*(\omega-\omega^\dagger)}^{q_*}\bigr)\;.
  \]
  Thus the assumption $\Phi(t) = o(t^{1/{q_*}})$ implies that,
  writing $\tilde{\omega} := \omega-\omega^\dagger$,
  \[
  \frac{\inner[b]{\tilde{\omega}}{J_{p_*}(\omega^\dagger)}}{\norm{\tilde{\omega}}}
  \le \frac{ \Phi\bigl(\hat{C}_{q_*}\norm{A^*\tilde{\omega}}^{q_*}\bigr)}{\norm{\tilde{\omega}}}
  \to 0
  \qquad
  \text{ as } \tilde{\omega} \to 0\;.
  \]
  This, however, is only possible, if $J_{p_*}(\omega^\dagger) = 0$,
  which is equivalent to stating that $\omega^\dagger = 0$.
  Thus $0 = A^*\omega^\dagger \in J_q(x^\dagger)$, and therefore $x^\dagger = 0$.
\end{proof}

\begin{lemma}
  Assume that $\mathcal{R}$ is two times Fr\'echet differentiable at $x^\dagger$
  and that $J_{p_*}(\omega^\dagger) \in \range A$.
  Then there exists $C > 0$ such that
  \[
  \inner[b]{\omega-\omega^\dagger}{J_{p_*}(\omega^\dagger)}
  \le C\bigl(\mathcal{D}_{x^\dagger}^*(A^*\omega;A^*\omega^\dagger)\bigr)^{1/2}
  \]
  for $\omega$ sufficiently close to $\omega^\dagger$.
\end{lemma}

\begin{proof}
  Because $\mathcal{R}$ is two times Fr\'echet differentiable at $x^\dagger$,
  there exists $c > 0$ such that
  \[
  \mathcal{R}(x) \le \mathcal{R}(x^\dagger) + \inner{A^*\omega^\dagger}{x-x^\dagger} +\frac{c}{2}\norm{x-x^\dagger}^2
  \]
  for every $x \in X$ sufficiently close to $x^\dagger$.
  Consequently,
  \[
  \mathcal{R}^*(A^*\omega) \ge \mathcal{R}^*(A^*\omega^\dagger) + \inner{A^*(\omega-\omega^\dagger)}{x^\dagger}
  +\frac{1}{2c}\norm{A^*(\omega-\omega^\dagger)}^2
  \]
  for $\omega$ sufficiently close to $\omega^\dagger$.
  This, however, is equivalent to stating that,
  locally around $\omega^\dagger$,
  \[
  \mathcal{D}_{x^\dagger}^*(A^*\omega;A^*\omega^\dagger) \ge \frac{1}{2c}\norm{A^*(\omega-\omega^\dagger)}^2\;.
  \]
  Now the assumption that $J_{p_*}(\omega^\dagger) \in \range A$
  is equivalent to the estimate
  \[
  \inner[b]{\omega-\omega^\dagger}{J_{p_*}(\omega^\dagger)} \le  \tilde{c}\norm{A^*(\omega-\omega^\dagger)}
  \]
  for some $\tilde{c} > 0$ and all $\omega \in Y^*$.
  Assembling these inequalities, the assertion follows.
\end{proof}

\section{Conclusion}

In this paper it is shown that the approach of variational inequalities
can be used for the derivation of higher order convergence rates
and is thus not restricted to the ``low rate world'' as has been surmised in~\cite{FleHof10}.
The basic idea for this generalisation is the formulation of the variational inequality
not for the primal Tikhonov functional, but rather for a dual functional.
By this approach we obtain the whole range of convergence rates
that have already been derived for quadratic regularisation and also
for convex regularisation on Banach spaces.
The main advantage of the usage of variational inequalities is
their comparative simplicity, in particular when used in conjunction with non-linear operators;
they have been introduced precisely for the study of non-linear ill-posed operator equations.
Thus it seems reasonable that the approach can be extended
also to the non-linear case without introducing too many artificial constraints
on the operator.

\end{document}